\documentclass[12pt,a4paper]{article}
\usepackage{amssymb}
\usepackage{graphicx,epsfig}
\usepackage{amssymb,amsfonts,amsmath,amsthm,dsfont,wasysym,pifont,stmaryrd}
\usepackage{epstopdf,yfonts}
\usepackage{graphicx,amssymb,amsfonts,epsfig,amsthm,a4,amsmath,url}
\usepackage[latin1]{inputenc}

\newtheorem{Thm}{Theorem}
\newtheorem{Prop}[Thm]{Proposition}
\newtheorem*{MainTh}{Theorem }
\newtheorem{Lem}[Thm]{Lemma}

\newtheorem{Cor}[Thm]{Corollary}
\theoremstyle{definition}
\newtheorem{Def}[Thm]{Definition}

\newtheorem{Rem}[Thm]{Remark}
\newtheorem{Ex}{Example}

\newcommand{\Z}{\mathbb{Z}}

\newcommand{\N}{\mathbf{N}}
\newcommand{\E}{\mathbb{E}}
\newcommand{\R}{\mathbb{R}}
\newcommand{\C}{\mathbb{C}}

\newcommand{\HH}{\mathcal{H}}
\newcommand{\bpr}{\noindent \textbf{Proof}: }

\newcommand{\epr}{\begin{flushright}{\large \ding{70}} \end{flushright}}

\renewcommand{\~}[1]{\overline{#1}}
\newcommand{\<}{\left\langle}
\renewcommand{\>}{\right\rangle}

\newcommand{\f}{\varphi}

\newcommand{\norm}{\trianglelefteqslant}

\newcommand{\Rr}{\mathcal{R}}
\newcommand{\Sum}[2]{\underset{#1}{\overset{#2}{\sum} }}

\title{Reduced 1-cohomology and relative property (T)}
\author{Talia Fern\'os, Alain Valette\footnote{With an appendix by Florian Martin and Alain Valette.}}
\date{\today}

\begin{document}
\maketitle
\baselineskip=16pt

\begin{abstract}
Shalom characterized property (T) in terms of the vanishing of all reduced first cohomology. We characterize group pairs having the property that the restriction map on all first reduced cohomology vanishes. We show that, in a strong sense, this is inequivalent to relative property (T).
\end{abstract}

\section{Introduction}

The celebrated Delorme-Guichardet Theorem (see e.g. \cite{BHV}) tells us that Kazhdan's property (T) for a locally compact, $\sigma$-compact group $G$, is equivalent to the vanishing of the first cohomology group $H^1(G,\pi)=Z^1(G,\pi)/B^1(G,\pi)$, for every unitary representation $\pi$. A deep result of Shalom \cite{ShaInv} says that, for $G$ a compactly generated group, this is still equivalent to the vanishing of the first {\it reduced} cohomology $\overline{H}^1(G,\pi)=Z^1(G,\pi)/\overline{B^1(G,\pi)}$,  for every unitary representation $\pi$. The closure is taken in the topology of uniform convergence on compact subsets of $G$.  Shalom used this to prove, among other things, that property (T) is an open property in the space of all finitely generated groups  \cite{ShaInv}.

If $A$ is a closed subgroup of the locally compact $\sigma$-compact group $G$, the relative property (T) for the pair $(G,A)$ can be characterized as the vanishing of the restriction map $Rest_G^A:H^1(G,\pi)\rightarrow H^1(A,\pi|_A)$ for every unitary representation $\pi$ of $G$ (see \cite{Jol}, \cite{Cornu}, \cite{FerTh}). So it is a natural question whether there is a ``relative'' version of Shalom's theorem, i.e. whether (possibly under extra assumptions on $G$ and $A$), the relative property (T) for the pair $(G,A)$ can be characterized by the vanishing of the restriction map on $\overline{H}^1$. Geometrically this vanishing means that every isometric action of $G$ on a Hilbert space almost has $A$-invariant points.

The aim of this paper is to show that the answer to this question is negative, in a very strong sense. We give a characterization for when the image $Rest_G^A:\overline{H}^1(G,\pi)\rightarrow\overline{H}^1(A,\pi|_A)$ is zero for every unitary representation of $G$. Using this characterization, we show that there are many amenable groups with non-compact subgroups which have this property.
On the other hand, if $G$ is amenable and $A$ is a non-compact closed subgroup of $G$, then the pair $(G,A)$ cannot have relative property (T).\footnote{We give a short proof of this fact in Proposition \ref{Haag} below.}

We consider two relative versions of Shalom's property $(H_T)$. Our main Theorem \ref{main} shows that the weaker one of these, along with a finite dimensional relative-fixed point property (relative property $F\E$) is equivalent to the vanishing of the above restriction map. We exhibit many classes of group pairs which satisfy these properties. As an example:

\begin{Thm}\label{nilpotent}
Let $G$ be a locally compact nilpotent group and $A=\overline{[G,G]}$. For every unitary representation $\pi: G \to U(\HH)$ the image of the restriction map $\overline{H}^1(G,\pi)\rightarrow\overline{H}^1(A,\pi|_A)$ is zero.
\end{Thm}

\subsection{Our Motivation}
The most classic example of a group pair with  relative property (T) is  $(SL_2(\Z)\ltimes \Z^2, \Z^2)$. There are various proofs of this fact \cite{Kazhdan}, \cite{Burger}, \cite{Valette1994},  \cite{Shalom1999}. However, all the proofs share two essential ingredients: they rely on the fact that $\Z^2$ is a normal, abelian subgroup and then apply the spectral theorem for the description of its unitary representations.

Nevertheless, the group $SL_2(\Z) \ltimes \Z^2$  is not very complicated. This lead people to ask if it is possible to show that $(SL_2(\Z)\ltimes \Z^2, \Z^2)$ has relative property (T) by a straightforward argument which does not appeal to the spectral theorem. The hope is that such an argument could generalize to other situations where the spectral theorem doesn't apply, e.g. if the subgroup is not abelian, the representation is not unitary, or if the vector space is a Banach space.

This paper is the outcome of our attempt to answer this question. As a by-product of our approach we obtain a proof of the following, which does not use the spectral theorem:

\begin{Cor}\label{SL2} Let $\Rr = \Z$ or $\R$, and let $\pi$ be a uniformly bounded representation of $\Rr^2\rtimes SL_2(\Rr)$ on a Hilbert space. Then the image of the restriction map $\overline{H}^1(\Rr^2\rtimes SL_2(\Rr),\pi)\rightarrow\overline{H}^1(\Rr^2,\pi)$ is zero.
\end{Cor}

The paper is organized as follows: Section 2 is a collection of properties about affine actions and the first cohomology. Section 3 is a short proof that amenable groups can not have relative property (T) with respect to non-compact subgroups. Section 4 is a discussion of Shalom's property $(H_T)$. In Section 5 we discuss the restriction map on the  reduced first  cohomology and prove our main theorem. Section 6 is devoted to demonstrating examples of group pairs which satisfy our main theorem using bounded generation. In Section 7, we study the connection between distorted subgroups and relative property $F\E$. Finally, the appendix is by Florian Martin and Alain Valette where they answer a question of Guichardet \cite{Guic}.

All of our groups are locally compact and second countable, representations are strongly continuous and Hilbert spaces are separable.

\medskip

{\bf Acknowledgements:} We thank B. Bekka, T. Gelander, Y. Glasner and Y. Shalom for some useful conversations and correspondence.
Also, we would like to thank Universit\'e d'Orl\'eans and in particular Indira Chatterji for their hospitality as this work was initiated during a common visit there.

\section{Affine actions and 1-cohomology}

We begin by discussing several facts about affine actions and 1-cohomology. We will make use of some of these, while the others are mentioned only for the sake of recording them, as they don't seem to appear elsewhere in the literature.

\subsection{Generalities}

 Let $G$ be a topological group and let $V$ be a topological $G$-module, i.e. a real or complex
 topological vector space endowed with a continuous linear representation $\pi:G\times V\rightarrow
V;\, (g,v) \mapsto \pi(g)v$. If $H$ is a closed subgroup we denote by $\pi\vert_H$ the
restriction of the representation $\pi$ to $H$, and by $V^H$ the set of
$H$-fixed points: $$V^H=\{v\in V\,\vert\, \pi(h)v=v,\, \forall h\in H\}.$$
We say that $V$ is a {\it Banach $G$-module} if $V$ is a Banach space
and $\pi$ is a representation of $G$ by isometries of $V$. A $G$-module
is {\it unitary} if $V$ is a Hilbert space and $\pi$ a unitary representation.\\ \\ We now
introduce the space of $1$-cocycles and $1$-coboundaries on $G$, and the $1$-cohomology
with coefficients in $\pi$:
\begin{enumerate}
\item[$\bullet$] $Z^1(G,\pi)=\{b:G\rightarrow V\,\mbox{continuous }\vert \, b(gh)=b(g)+\pi(g)b(h),\\
\hspace*{23mm}\forall g,h\in G\}$;
\item[$\bullet$] $B^1(G,\pi)=\{b\in Z^1(G,\pi)\vert \, \exists v\in V: \, b(g)=\pi(g)v-v,\forall g\in G\}$;
\item[$\bullet$] $H^1(G,\pi)=Z^1(G,\pi)/B^1(G,\pi)$.
\end{enumerate}

Note the following formula, which follows by iterating the cocycle relation; for every $g,h\in G$:
\begin{equation}\label{cocy} b(ghg^{-1})=(1-\pi(ghg^{-1}))b(g) + \pi(g)b(h).
\end{equation}
This equation will be used repeatedly throughout this paper.

There is a well known dictionary between $1$-cocycles of $Z^1(G,\pi)$ and (continuous)
affine actions of $G$ on $V$. If $b:G\rightarrow V$ is a continuous map then:\\
\begin{center}
\begin{tabular}{ccp{5cm}}$b\in Z^1(G,\pi)$&$\Leftrightarrow$&$ \alpha(g)v=\pi(g)v+b(g)$ defines an affine
action with linear part given by $\pi$\end{tabular}\end{center}
Here, $B^1(G,\pi)$ corresponds exactly to actions with a global fixed point, i.e. those actions conjugate to the module action by a translation of $V$. We will make repeated use of the fact that, if $\pi$ is a unitary representation on a Hilbert space, and $b\in Z^1(G,\pi)$, we have $b\in B^1(G,\pi)$ if and only if $b$ is bounded on $G$ (see \cite{BHV}, Proposition 2.2.9).

\medskip


If $A$ is a closed normal subgroup of $G$, and $\pi$ is a $G$-representation, we denote by $\pi^A$ the representation of $G/A$ on $V^A$.




The following result (``Lyndon-Hochschild-Serre sequence'') is well-known (see
e.g. Corollary 6.4 of Chapter VII in \cite{BrownCoho}) and usually
proved using the Hochschild-Serre spectral sequence in group
cohomology.\footnote{For a proof without spectral sequences, see 8.1
in Chapter 1 of \cite{Gui}.}

\begin{Prop}\label{exactNormal}
\begin{enumerate}
    \item  [1)]  There is an exact sequence
    $$0\rightarrow
    H^{1}(G/A,\pi^{A})\stackrel{i_{*}}{\rightarrow}H^{1}(G,\pi)
    \stackrel{Rest_{G}^{A}}{\rightarrow}H^1(A,\pi\vert_A)^{G/A}$$
    where $i:V^{A}\rightarrow V$ denotes the inclusion;

    \item  [2)] If $V^A=0$, then the restriction map $$Rest_G^A:H^1(G,\pi)\rightarrow
H^1(A,\pi\vert_A)^{G/A}$$ is an isomorphism.
\end{enumerate}
\end{Prop}

\subsection{Reduced 1-cohomology}

In this section, we stick to Banach $G$-modules.

\begin{Def}
Let $G$ be a locally compact second countable group and $V$ be a
Banach $G$-module. The space of $1$-cocycles is a Fr\'echet space when endowed with the
topology of uniform convergence on compact subsets of $G$. The $1$-reduced cohomology space with
coefficients in $V$ is $$\overline{H^1}(G,\pi)=Z^1(G,\pi)/\overline{B^1(G,\pi)}.$$ In terms of
affine actions the elements of $\overline{B^1(G,\pi)}$, called almost coboundaries,
correspond to affine action $\alpha$ which are almost conjugate by a translation to the
unitary action, in the sense that for every $\varepsilon>0$ and for every compact subset
$K$ of $G$ there exists a translation $t_v$ of $V$ such that, for
every $g\in K$, the norm of the translation
$\alpha(g)-t^{-1}_v\circ\pi(g)\circ t_v$ is less than $\varepsilon$.
\end{Def}

\medskip

If $(V,\pi)$ is a Banach $G$-module, and $\alpha$ is an affine action with linear part
$\pi$, we say that $\alpha$ almost has fixed points if for every $\varepsilon>0$ and for
every compact subset $K$ of $G$ there exists a $v\in V$ such that
$\sup_K\|\alpha(g)v-v\|<\varepsilon$. This is equivalent to be almost conjugate  by a
translation to the unitary action. So $\overline{H^1}(G,\pi)=0$ if and only if every affine
action with linear part $\pi$ almost has fixed points. The following lemma is then clear:

\begin{Lem}\label{inductivelim} Let $G$ be the increasing union of a family $(G_n)_{n\geq 1}$ of open subgroups, and let $(\pi,V)$ be a Banach $G$-module. If $\~H^1(G_n,\pi|_{G_n})=0$ for every $n\geq 1$, then $\~H^1(G,\pi)=0$.
\end{Lem}

\begin{Rem}\label{finitedim} Observe that, if $V$ is finite-dimensional, then so is $B^1(G,\pi)$ (as the image of $V$ under the linear map $V\rightarrow Z^1(G,\pi): v\mapsto (g\mapsto \pi(g)v-v)$), so that $B^1(G,\pi)$ is closed in $Z^1(G,\pi)$, i.e. $\overline{H^1}(G,\pi)=H^1(G,\pi)$ in this case.
\end{Rem}





\section{The Haagerup Property vs. Relative Property (T)}
Recall that a locally compact group $G$ has the {\it Haagerup property} if there exists a unitary representation $\sigma$ of $G$ and $b\in Z^1(G,\sigma)$ such that $b$ is proper (i.e. $\lim_{g\rightarrow\infty} \|b(g)\|=+\infty$). It was proved in \cite{BCV} that $\sigma$-compact amenable groups have the Haagerup property.

\begin{Prop}\label{Haag} Let $G$ be a group with the Haagerup property, and $A$ a closed, non-compact subgroup of $G$. There exists a unitary representation $\sigma$ of $G$ such that $Rest_G^A: H^1(G,\sigma)\rightarrow H^1(A,\sigma|_A)$ is non-zero.
\end{Prop}

\bpr Take a representation $\sigma$ of $G$ admitting a proper $b\in Z^1(G,\sigma)$. Since $A$ is closed and not compact, $b|_A$ is unbounded, so that the class of $b|_A$ cannot be zero in $H^1(A,\sigma|_A)$.\epr

When $G$ is amenable, the representation $\sigma$ in Proposition \ref{Haag} can be made explicit: the proof in \cite{BCV} reveals that the direct sum of countably many copies of the regular representation of $G$, does the job.

\section{Property $(H_T)$}

As an appetizer to our main results, we give the following theorem, which exemplifies our approach.

Recall from \cite{ShaActa} that a locally compact group $G$ has property $(H_T)$ if every unitary representation $\pi$ of $G$ with $\overline{H}^1(G,\pi)\neq 0$, has non-zero invariant vectors. It is known that locally compact nilpotent groups have property $(H_T)$ (Corollary 5.1.3 in \cite{ShaActa}), as well as lamplighter groups $F\wr\Z$, where $F$ is any finite group (Theorem 5.2.1 in \cite{ShaActa}). And, since property (T) groups have all $\overline{H}^1(G,\pi)= 0$, these trivially also have property $(H_T)$. The following result extends Theorem \ref{nilpotent} from the Introduction.


\begin{Thm}\label{HT} Let $G$ have property $(H_T)$ and $A$ be a closed subgroup contained in $\overline{[G,G]}$. Then for every unitary representation $\pi$ of $G$, the restriction map $Rest_G^A:\overline{H}^1(G,\pi)\rightarrow\overline{H}^1(A,\pi|_A)$ is zero.
\end{Thm}

\begin{proof}

Let ${\cal H}$ be the Hilbert space of $\pi$. Let $\pi^{\perp}$ be the representation on the orthogonal complement of ${\cal H}^G$. Then $\pi=\pi^G\oplus\pi^{\perp}$ and $\overline{H}^1(G,\pi)=\overline{H}^1(G,\pi^G)\oplus\overline{H}^1(G,\pi^{\perp})$. Since $\pi^{\perp}$ has no non-zero invariant vector, by property $(H_T)$ the second summand is zero. Now
$$\overline{H}^1(G,\pi^G)=H^1(G,\pi^G)=Z^1(G,\pi^G)=\hom(G,{\cal H}^G),$$
the continuous homomorphisms from $G$ to the additive group of ${\cal H}^{G}$. Since $A$ is contained in $\overline{[G,G]}$, every such homomorphism is zero on $A$, which concludes the proof.
\end{proof}

\begin{Ex} Take for $G$ the 3-dimensional Heisenberg group, and for $A$ its center: then both Proposition \ref{Haag} and Theorem \ref{HT} apply to the pair $(G,A)$.
\end{Ex}

\section{Restriction in reduced 1-cohomology}

\subsection{Relative Properties $(H_T)$}

Given Theorem \ref{HT}, we explore some relative variants of property $(H_T)$, and their connection to the vanishing of restriction maps on reduced cohomology.

\begin{Def}
Let $A \leq G$ be a closed  subgroup. We say that the pair $(G,A)$ has relative $(H_T)$ if every unitary $G$-representation $\pi$ with $\~H^1(G, \pi) \neq 0$ has nontrivial $A$-invariant vectors.
\end{Def}

\begin{Def}
Let $A\leq G$ be a closed subgroup. We say that $(G,A)$ has weak-relative $(H_T)$ if any unitary $G$-representation $\pi$ with a non-vanishing restriction map $\~H^1(G, \pi) \to \~H^1(A, \pi|_A)$ must have nontrivial $A$-invariant vectors.
\end{Def}

\noindent
The following lemma is straightforward.

\begin{Lem} Let $A$ be a closed subgroup in $G$. Consider the following 4 properties:
\begin{enumerate}
\item [1)] The group $G$ has property $(H_T)$;
\item [2)] The group $A$ has property $(H_T)$;
\item [3)] The pair $(G,A)$ has relative $(H_T)$;
\item [4)] The pair $(G,A)$ has weak relative $(H_T)$.
\end{enumerate}
Then $(1)\Rightarrow (3)\Rightarrow (4)$ and $(2)\Rightarrow (4)$.
\hfill$\square$
\end{Lem}



In general, it is FALSE that $(2)\Rightarrow (3)$. We give the following example:
\begin{Ex} Let $\Gamma$ be an infinite discrete group with property (T), and let $F$ be a non-trivial finite abelian group. Let $G=F\wr\Gamma$ be the corresponding wreath product; set $A=\bigoplus_{\Gamma} F$, an abelian normal subgroup in $G$. Then $A$ has property $(H_T)$, as it is abelian, while the pair $(G,A)$ does not have relative $(H_T)$, as a consequence of Theorem \ref{wreath} in the Appendix.
\end{Ex}

\begin{Lem}
Suppose that $(G,A)$ has relative property (T) and $A$ is normal in $G$. Then $(G,A)$ has relative property $(H_T)$.
\end{Lem}

\begin{proof} By contrapositive, if $(G,A)$ fails to have relative $(H_T)$ then there is a unitary $G$-representation  $\pi$ with $\~H^1(G, \pi) \neq 0$ but for which the only $A$-invariant vector is 0. Observe that  $\~H^1(G, \pi) \neq 0$ implies that $H^1(G, \pi) \neq 0$.

By Proposition \ref{exactNormal} the restriction map $H^1(G, \pi) \to H^1(A, \pi|_A)$ is injective and has non-trivial image. This of course means that $(G,A)$ does not have relative property (T).
\end{proof}

The next proposition can be seen as the relative version of Theorem 4.2.1 (1) of Shalom's \cite{ShaActa}.

\begin{Prop}\label{baby} Let $A\norm G$ be a closed normal subgroup of $G$. The following are equivalent:
\begin{enumerate}\label{wrelHT}
\item [a)] The pair $(G,A)$ has weak relative $(H_T)$;
\item [b)] Every unitary irreducible representation $\pi$ of $G$ with a non-vanishing restriction map $\~H^1(G, \pi) \to \~H^1(A, \pi|_A)$, factors through $G/A$.
\end{enumerate}
\end{Prop}

\begin{proof}
$(a)\Rightarrow (b)$ If $\pi$ is irreducible with a non-vanishing restriction map $\~H^1(G, \pi) \to \~H^1(A, \pi|_A)$, then since $(G,A)$ has weak relative $(H_T)$, the space ${\cal H}^A$ of $\pi(A)$-fixed vectors is non-zero. As $A\norm G$, that subspace is $\pi(G)$-invariant, so by irreducibility ${\cal H}^A={\cal H}$. In other words $\pi$ is trivial on $A$.

$(b)\Rightarrow (a)$ Let $\pi$ be a unitary representation with a non-vanishing restriction map $\~H^1(G, \pi) \to \~H^1(A, \pi|_A)$. As G is second countable, by decomposition theory, we find a measure space $(X,{\cal B},\mu)$ and a measurable field $(\rho_x)_{x\in X}$ of irreducible representations of $G$, such that $\pi =\int^\oplus_X \rho_x\,d\mu(x)$. Let $b\in Z^1(G,\pi)$ be such that the class of $b\vert_A$ is non zero in $\~H^1(A,\pi|_A)$. Write $b=\int^\oplus_X b_x\,d\mu(x)$, where $b_x\in Z^1(G,\rho_x)$ for every $x\in X$. By Proposition 2.6 in Chapter III of \cite{Gui}, there exists $E\in{\cal B}$, with $\mu(E)>0$, such that the class of $b_x\vert_A$ is non-zero in $\~H^1(A,\rho_x|_A)$ for almost every $x\in E$. By assumption, every such representation $\rho_x$ is trivial on $A$. Consider then the subspace ${\cal K}$ of ${\cal H}_\pi$ consisting of measurable vector fields $\xi=\int^\oplus_X \xi_x\,d\mu(x)$ such that $\xi_x=0$ for almost every $x\in X\backslash E$. Then ${\cal K}$ is a non-zero closed subspace of ${\cal H}_\pi$, whose elements are $\pi(A)$-fixed.
\end{proof}

\subsection{The Main Theorem}

Let $\E^n$ denote the $n$-dimensional Euclidean space, i.e. $\R^n$ with the $\ell^2$ norm. We consider the isometry group $Isom(\E^n)$ which is isomorphic to the motion group $O(n)\ltimes \E^n$.

\begin{Def}\label{relFE}
A group pair $(G,A)$ has relative property $F\E$ if every isometric $G$ action on a finite dimensional Euclidean space has an $A$-fixed point.
\end{Def}

If the pair $(G,A)$ has the property that the restriction map $\overline{H}^1(G,\pi)\rightarrow\overline{H}^1(A,\pi\vert_A)$ is zero for every unitary representation $\pi$ of $G$, then $(G,A)$ has relative property $F\E$, as reduced cohomology coincides with unreduced one for finite-dimensional representations (see Remark \ref{finitedim}).

If one recalls the fact that $$O(n) \ltimes \E^n \hookrightarrow U(n) \ltimes \C^n \hookrightarrow O(2n) \ltimes \E^{2n},$$ then the following is elementary:

\begin{Lem}\label{CvsR}
The following conditions are equivalent:
\begin{enumerate}
\item The group pair $(G,A)$ has relative property $F\E$;
\item For every finite dimensional orthogonal representation $\f : G \to O(n)$ and $b: G \to \E^n$ a cocycle over $\f$ then, $b|_A$ is bounded.
\item For every finite dimensional unitary representation $\pi : G \to U(n)$ and $b: G \to \C^n$ a cocycle over $\f$ then, $b|_A$ is bounded.
\end{enumerate}
\end{Lem}


\begin{Thm}\label{main}
Let $A \norm G$ where $A$ is a compactly generated closed subgroup. The following are equivalent:
\begin{enumerate}
\item $(G, A)$ has both weak-relative property $(H_T)$ and relative $F\E$.
\item For every unitary $G$-representation $\pi$ the restriction map  $\~H^1(G, \pi) \to \~H^1(A, \pi|_A)$ is zero.
\end{enumerate}
\end{Thm}

\begin{proof} $(1) \implies (2)$: By normality of $A$ we have a decomposition of the representation $\pi= \pi^A\oplus \pi^\perp$ where $\pi^A$ and $\pi^\perp$ are the $G$-representations corresponding to the $A$-invariant vectors and its orthogonal complement. It is simple to show that the compact-open topology on $Z^1(G, \pi)$ gives the decomposition $\~B^1(G, \pi) = \~B^1(G, \pi^A) \oplus \~B^1(G, \pi^\perp)$.

We therefore have that $\~H^1(G, \pi)=\~H^1(G, \pi^A)\oplus \~H^1(G, \pi^\perp)$. We must show that the restriction map vanishes on each of these components. By the assumption that $(G,A)$ has w-relative $(H_T)$, the restriction map on $\pi^\perp$ vanishes.

Consider now $b\in Z^1(G, \pi^A)$. Observe that $b|_A$ is a continuous homomorphism from the compactly generated group $A$ to the abelian group $(\HH^A, +)$. Hence $b|_A$ factors through the compactly generated, locally compact abelian group $B:=A/\overline{[A,A]}$. By the structure theorem for locally compact abelian groups (see Theorem 2.4.1 in \cite{Rudin}), $B$ contains an open subgroup $U$ of the form $K\times\R^m$, where $K$ is a compact abelian group. In our case $B/U$ is a finitely generated abelian group. Since $b|_K\equiv 0$, we see that $b(A)$ is contained in a finite dimensional subspace of $\HH$. So let ${\cal K}$ be the linear span of $b(A)$ in $\HH$. By formula (\ref{cocy}), we have $\pi(g)b(a) = b(g\cdot a)$ for every $g\in G,\,a\in A$, which means that $b(A)$ is $\pi(G)$ invariant, hence also ${\cal K}$ is. Denote by $\sigma$ the restriction of $\pi^A$ to ${\cal K}$, and by $\sigma^\perp$ the orthogonal of $\sigma$ in ${\cal H}^A$, so that $\pi^A=\sigma\oplus\sigma^\perp$. Let $b=b^0\oplus b^\perp$ be the corresponding decomposition of $b$. Since $b(A)\subset{\cal K}$, we have $b^\perp|_A=0$; so it remains to show that $b^0|_A$ is a co-boundary; but since $\sigma$ is a finite-dimensional unitary representation of $G$, this follows from relative property $F\E$.

\medskip

$(2) \implies (1)$: Clear in view of the sentence following Definition \ref{relFE}. 
\end{proof}

\noindent
{\bf{Remark 1}:} If $G$ has property $(H_T)$ and $A$ is a closed normal subgroup contained in $\~{[G,G]}$, it is easy to see that the pair $(G,A)$ has both the weak-relative property $(H_T)$ and the relative property $F\E$ (compare with Theorem \ref{HT}).

\noindent
{\bf{Remark 2:}} Groups whose first reduced cohomology always vanish, were considered by Shalom \cite{ShaInv}, in the compactly generated case, and were also studied by de Cornulier for non-compactly generated locally nilpotent groups \cite{deCor}. This property of the vanishing of all first reduced cohomology is also called property $\~{FH}$. Hence, the group-pair property we study here, where the restriction map on the first reduced cohomology always vanishes, can be said to be relative property  $\~{FH}$. However, we felt that there are already many definitions in this article and that this one would not improve our exposition. Nevertheless, we rephrase our main theorem in this terminology:

\begin{MainTh}[\ref{main}']
Let $A \norm G$ where $A$ is a compactly generated closed subgroup. The group pair $(G,A)$ has relative property $\overline{FH}$ if and only if it has both relative property $F\E$ and weak relative $(H_T)$.
\end{MainTh}

\section{Bounded Generation}

Again, let us assume that $A\norm G$.  Let $a^g = gag^{-1}$ denote the $g$-conjugate of $a$ and $a^G$ the full $G$-orbit.

\begin{Def}
Let $A \norm G$ be a closed normal subgroup. We say that $A$ is cc-boundedly generated in $G$ if $A$ is the product of finitely many $G$-conjugacy classes, i.e. for some $k$, there exists $a_1, \dots, a_k \in A$ so that $A = a_1^G \cdots a_k^G$.
\end{Def}

\begin{Prop}\label{bddgen}
If $A \norm G$ is cc-boundedly generated and $(G, A)$ has weak-relative $(H_T)$ then every restriction map in reduced cohomology vanishes.
\end{Prop}

\begin{proof}
Let $\pi$ be a unitary representation of $G$. Then as before, normality gives us $$\~H^1(G, \pi) = \~H^1(G, \pi^A) \oplus \~H^1(G, \pi^\perp).  $$

The second summand of the restriction map must of course be zero. Let us show the same for the first. Since $\overline{H}^1(A,\pi^A)=H^1(A,\pi^A)=\hom(A,{\cal H}^A)$, it is enough to see that the restriction map $H^1(G,\pi^A)\rightarrow H^1(A,\pi^A|_A)$ is zero. To see that, consider $b \in Z^1(G, \pi^A)$. By equation (\ref{cocy}) we have for any $g\in G$ and $a\in A$

$$b(a^{g})=\pi^A(g)b(a).$$

Let $a_1, \dots, a_k \in A$ such that $A = a_1^G \dots a_k^G$ and $M =  \Sum{i=1}{k} \|b(a_i)\|$. Take $a\in A$ an arbitrary element. Then, there exists $g_i \in G$ so that $a = a_1^{g_1}\dots a_k^{g_k}$. Using the previous calculation, we have

\begin{eqnarray*}
\|b(a)\| &\leq& \Sum{i=1}{k} \|b(a_i^{g_i})\|\\
&=& \Sum{i=1}{k} \|\pi^A({g_i})b(a_i)\|\\
&=& \Sum{i=1}{k} \|b(a_i)\|=M,
\end{eqnarray*}
\noindent
which is uniformly bounded over $A$.
\end{proof}


\subsection{Amenable normal subgroups}

In this subsection, we give an alternative version of Proposition \ref{bddgen} for uniformly bounded representations.

\begin{Lem}\label{Semidir} Let $A\norm G$ be cc-boundedly generated in $G$ and $\pi$ be a uniformly bounded representation of $G$ on a Banach space $B$, which is trivial on $A$. Any 1-cocycle $b\in Z^1(G,\pi)$ is bounded on $A$.
\end{Lem}

\begin{proof} 
The proof is the same as in Proposition \ref{bddgen}.
\end{proof}
\medskip

\begin{Thm}\label{UB} Let $A\norm G$. Assume that $A$ is cc-boundedly generated in $G$, $A$ is amenable and $A$ has property $(H_T)$. Let $\pi$ be a uniformly bounded representation of $G$ on a Hilbert space ${\cal H}$. Then the restriction map $Rest_G^A: \overline{H}^1(G,\pi)\rightarrow\overline{H}^1(A,\pi|_A)$ is zero.
\end{Thm}

\begin{proof} As $A$ is amenable, we may assume that $\pi|_A$ is a unitary representation. Let ${\cal H}^{\perp}$ be the orthogonal complement of the space ${\cal H}^A$ of $\pi(A)$-fixed vectors. As $A$ is normal in $G$, the space ${\cal H}^A$ is clearly $\pi(G)$-invariant.

We claim that ${\cal H}^{\perp}$ is $\pi(G)$-invariant. To see it, recall from \cite{BFGM} that, if $\rho$ is a uniformly bounded representation of a locally compact group $L$ on a super-reflexive Banach space $B$, there is a decomposition $B=B^{\rho(L)}\oplus B'$, where $B'$ is the annihilator of $(B^*)^{\rho^*(L)}$, the space of fixed vectors for the contragredient representation $\rho^*$ on the dual space $B^*$ (see Proposition 2.6 in \cite{BFGM}); moreover this decomposition is preserved by the normalizer of $\rho(L)$ in $GL(B)$ (Corollary 2.8 in \cite{BFGM}). We apply this with $B={\cal H},\,L=A$ and $\rho=\pi|_A$: as $\rho$ is unitary it identifies with its contragredient, so $B'={\cal H}^{\perp}$. As $A$ is normal in $G$, the claim follows.

As usual, denote by $\pi^A$ (resp. $\pi^{\perp}$) the restriction of $\pi$ to ${\cal H}^A$ (resp. ${\cal H}^{\perp}$). Then
$$\overline{H}^1(G,\pi)=\overline{H}^1(G,\pi^A)\oplus\overline{H}^1(G,\pi^{\perp}).$$

By lemma \ref{Semidir}, the restriction map $\overline{H}^1(G,\pi^A)\rightarrow\overline{H}^1(A,\pi^A|_A)$ is zero. On the other hand, since $A$ has property $(H_T)$ and $\pi^{\perp}|_A$ is unitary and has no non-zero fixed vector, we have $\overline{H}^1(A,\pi^{\perp}|_A)=0$.\end{proof}

\begin{Ex} Let $D=(\R_+^*)^n$ be identified with the diagonal subgroup of $GL_n(\R)$. Let $H$ be a closed subgroup of $D$ which projects surjectively onto each factor. Then $\R^n$ is boundedly generated by the $H$-orbits of the $2n$ vectors $\pm e_1,...,\pm e_n$ (where $e_1,...,e_n$ is the standard basis of $\R^n$). So both Proposition \ref{Haag} and Theorem \ref{UB} apply to the pair $(\R^n\rtimes H,\R^n)$.
\end{Ex}

We now move on to the proof of Corollary \ref{SL2}. We take the opportunity to generalize it.

\begin{Prop} Let $\Rr$ be a locally compact topological ring with unit and let $EL_2(\Rr)$ denote the group generated by 2-by-2  elementary matrices with entries in $\Rr$. Consider the group $G= EL_2(\Rr)\ltimes \Rr^2$, with abelian subgroup $A=\Rr^2$. For every uniformly bounded representation $\pi$ of $G$ on a Hilbert space, the restriction map $Rest_G^A: \overline{H}^1(G,\pi)\rightarrow\overline{H}^1(A,\pi|_A)$ is zero.
\end{Prop}\label{loccompring}

\begin{proof} The subgroup $A$ is clearly closed and normal inside $G$. As $A$ is abelian, to apply Theorem \ref{UB} we need only show that it is cc-boundedly generated in $G$. We claim that $A = (0,1)^G +(1,0)^G$.

To this end, consider $(a, b) \in A$. Let $g_1 = \left(\begin{array}{cc}1 & a-1 \\0 & 1\end{array}\right)$ and $g_2 = \left(\begin{array}{cc}1 & 0 \\b-1 & 1\end{array}\right)$. Then

$$(0,1)^{g_1} = g_1 \cdot (0,1) = (a-1, 1) $$ and $$ (1,0)^{g_2} = g_2 \cdot (1,0) = (1,b- 1).$$

Therefore, $(a, b) = (0,1)^{g_1} + (1,0)^{g_2}$ and this completes the proof.
\end{proof}
\medskip

Following \cite{BFGM}, say that a locally compact group $G$ has property $(\overline{F}_{\cal H})$ if every uniformly bounded affine action (i.e. an affine action whose linear part is a uniformly bounded representation) on a Hilbert space, has a fixed point. It is an unpublished result by Y. Shalom (see however Remark 1.7.(3) in \cite{BFGM}) that higher rank simple Lie groups and their lattices have property $(\overline{F}_{\cal H})$. As a consequence, for $n\geq 3$ and for $\Rr=\R$ or $\Z$, the semi-direct product $G= SL_n(\Rr)\ltimes \Rr^n$  has property $(\overline{F}_{\cal H})$. To see it, let $\pi$ be a uniformly bounded representation of $G$ on a Hilbert space, with $\|\pi(g)\|\leq C$ for every $g\in G$, and $b\in Z^1(G,\pi)$; then for $h\in SL_n(\Rr)$ and $v\in \Rr^n$, we have $\|b(h\cdot v)\|\leq 2C\|b(h)\|+C\|b(v)\|$ by formula (\ref{cocy}). By Shalom's result, $b$ is bounded on $SL_n(\Rr)$, so that $b$ is bounded on every $SL_n(\Rr)$-orbit in $\Rr^n$. The conclusion follows as in the above proof.

\section{Distortion and Relative Property $F\E$.}

We discuss some examples of group pairs with relative property $F\E$. In this section, we will consider discrete finitely generated groups.

Let $G$ be finitely generated by $S = S^{-1} \subset G$. It is well known that this gives rise to a length function $l_S: G \to \N\cup\{0\}$ with the following properties: Let $g, h\in G$ and $1$ denote the identity. Then
\begin{itemize}
\item $l_S(1) = 0$,
\item $l_S(g) = l_S(g^{-1})$,
\item and $l_S(gh) \leq l_S(g)+l_S(h)$.
\end{itemize}

It is easy to see that, if $\pi$ is a unitary representation of $G$ and $b\in Z^1(G,\pi)$, then $\|b(g)\|\leq K.l_S(g)$, where $K=\max_{s\in S} \|b(s)\|$.

We now have a brief look at subgroup distortion, which is an interesting area of study in it's own right, see \cite{LMR}.

\begin{Def}
An element $a\in G$ is said to be undistorted if the word length of $a^n$ grows linearly in $n$. It is distorted otherwise.
\end{Def}

It is a simple exercise, based on the triangle inequality, that being (un)distorted is a conjugacy invariant.
We now recall a more classic notion of bounded generation:

\begin{Def} A subgroup $A\subset G$ is boundedly generated by finitely many distorted elements in $G$ if there are finitely many $a_1, \dots, a_k \in A$ so that $A=\< a_1\> \cdots \<a_k\>$ and $a_i$ is distorted in $G$ for each $i=1,...,k$.
\end{Def}

The main result of this section is:

\begin{Prop}\label{dist} Let $A$ be a subgroup of $G$. Each of the following two conditions imply relative property $F\E$.
\begin{enumerate}
\item $A$ is boundedly generated by finitely many distorted elements in  $G$.
\item $A$ is a normal subgroup, $(G,A)$ has weak relative $H_T$ and $A$ is generated by distorted elements in $G$.
\end{enumerate}
\end{Prop}

We remark that item 1 has  an analogous formulation in the case of isometric actions on complete CAT(0) spaces. In such a formulation, it would be a relative version of \cite[Lemma 8.1]{CapMon}.


\begin{Lem}\label{LMR} Let $\phi$ be a finite-dimensional unitary representation of $G$, and $a\in G$ be a distorted element. There exists $m\geq 1$ such that, for every affine isometric action $\alpha$ of $G$ with linear part $\f$, the isometry $\alpha(a)$ has finite order $m$.
\end{Lem}

\begin{proof} By the proof of Proposition 2.4 in \cite{LMR}, the image of $a$ in any finite-dimensional representation $\rho$ of $G$, is virtually unipotent; this means that, for some $k\geq 1$, the matrix $\rho(a)$ is unipotent. Since the only unipotent element in $U(N)$ is the identity, we see that $\f(a)$ has finite order $m$. Let then $H$ be the subgroup generated by $a^m$, and let $b\in Z^1(G,\f)$ be the translation part of $\alpha$; then $b|_H$ is a homomorphism from $H$ to $\C^N$ and, for every $n\geq 1$, we have: $n\|b(a^m)\|=\|b(a^{mn})\|\leq K\cdot l_S(a^{mn})$. Since $a$ is distorted, this implies that $b(a^m)=0$, i.e $\alpha(a)$ has order $m$.
\end{proof}

\begin{proof}[Proof of Proposition \ref{dist} Item (1)]
Let us fix  $\f : G \to U(n)$ and $b \in Z^1(G, \f)$, and write $A= \<a_1\> \cdots \<a_k\>$, with $a_1,...,a_k$ distorted in $G$. By Lemma \ref{LMR}, the cocycle $b$ is bounded on each subgroup $\<a_i\>$, so it is bounded on $A$.
\end{proof}

\begin{proof}[Proof of Proposition \ref{dist} Item (2)] Again fix $\f$ a finite-dimensional unitary representation of $G$ on $V$, and $b\in Z^1(G,\f)$. As before, write $\f=\f^A\oplus\f^\perp$; let $b=b_0\oplus b^\perp$ be the corresponding decomposition of $b$. By the weak relative property $(H_T)$, the cocycle $b^\perp$ is bounded on $A$. On the other hand, $b_0\vert_A$ is a homomorphism from $A$ to the additive group of $V^A$. By lemma \ref{LMR} applied to the representation $\f^A$, the homomorphism $b_0$ is bounded, hence zero, on the cyclic subgroup generated by any distorted element in $A$. Since $A$ is generated by distorted elements, $b_0\vert_A=0$.
\end{proof}

\subsection{On being generated by distorted elements}

Since finitely generated linear groups are virtually torsion free, to each affine isometric action $\alpha: G \to Isom(\E^n)$ we can associate a finite index subgroup $G_\alpha$ so that $\alpha(G_\alpha)$ is torsion free. In Lemma \ref{LMR}, we noted that any such $\alpha$ maps distorted elements to torsion elements. Hence any distorted element in $G_\alpha$ is in the kernel of $\alpha$.

Assume $A$ is generated by distorted elements in $G$. If it were true that this would imply that every finite index subgroup were generated by distorted elements as well, then Proposition \ref{dist} would hold under the mere assumption that $A$ is generated by distorted elements.

However, here is an example of a finitely generated group $G$ and a subgroup $A$ generated by finitely many distorted elements, such that the pair $(G,A)$ does not have relative property $F\E$.
Let $$H = \<x,y,z:[x,y]=z, [y,z]=[z,x]=1 \>$$ denote the discrete Heisenberg group. We will use the fact that $z$ is distorted in $H$. (Indeed, $[x^n,y^n]= z^{n^2}$ so the length of $z^{n^2}$ is at most $4n$.)

First observe that $H$ admits the dihedral group $D_4$ (the isometries of a square) as a quotient: Let $s, t$ be the reflections in the plane $\R^2$ about two lines which intersect and have an angle of $\pi/4$. Also let $c$ be the rotation of angle $\pi$ about their intersection point, and observe that it is central. Then $x\mapsto s$, $y\mapsto t$, $z\mapsto c$ is a homomorphism of $H$ onto $D_4$.

Let $G$ be the free product of two copies $H_1$, and $H_2$ of $H$, and let $A$ be the subgroup generated by the corresponding elements $z_1$, $z_2$ (generating the center in each copy). As a group $A$ is the free group on 2 generators; clearly $A$ is generated by distorted elements.

We now claim that $(G,A)$ does not have relative $F\E$. To see it, construct $D_4$ in two ways, around two distinct points $P_1$, $P_2$ in the plane (defining as above $s_1$, $t_1$, $c_1$ and $s_2$, $t_2$, $c_2$). Map $H_1$ to the first $D_4$, and $H_2$ to the second one. This gives us an isometric action of $G$ on the plane. Now the image of the product $z_1z_2\in A$ is the product $c_1c_2$, which is a non-zero translation in the plane, so it has no fixed point.

\section*{Appendix, by Florian Martin and Alain Valette: {\it A question of Guichardet}}

Here is a quotation from Guichardet (p. 319 in \cite{Guic}), who considered
the case of a unitary $G$-module $V$ with $V^{N}=0$ (where $N$ is a normal subgroup of $G$): {\it ``On ignore
par exemple si $\overline{B}^{1}(N,\pi|_N)=Z^{1}(N,\pi|_N)$ implique
$\overline{B}^{1}(G,\pi)=Z^{1}(G,\pi)$''}. We shall see below that lots
of wreath products provide a negative answer to Guichardet's question,
even with $N$ abelian. This shows that there is no analogue of Proposition \ref{exactNormal} in reduced cohomology.
\medskip

Recall that the
wreath product of two discrete groups $H,\Gamma$, denoted by $H\wr \Gamma$, is the semi-direct
product of $\oplus_{\Gamma}H$ by $\Gamma$, acting by shifting the indices.

\begin{Thm}\label{wreath}
Let $H,\Gamma$ be two (non-trivial) countable groups with property (T), with $\Gamma$
infinite. Set $N=\oplus_{\Gamma}H$ and $G=H\wr\Gamma$. There exists a
unitary, irreducible representation $(\pi,V)$ of $G$ with $V^N=0$, $\overline{H^1}(G,\pi)\not=0$ and $\overline{H^1}(N,\pi\vert_N)=0$. In particular, the pair $(G,N)$ does not have relative $(H_T)$.
\end{Thm}

{\bf Proof:}
Since $\Gamma$ is infinite, it is known that $G$ does not have property (T) (see
Corollary 1 in \cite{CMV}). On the other hand, $G$ is finitely generated. By a result of Shalom
(\cite{ShaInv}, Thm 0.2 and remark p.30) there exists a unitary irreducible $G$-representation $(\pi,V)$
such that $\overline{H^1}(G,V)\not=0$.

Let $(F_n)_{n\geq 1}$ be an
increasing sequence of finite subsets of $\Gamma$, with $\Gamma=\bigcup_{n\geq 1}F_n$. Set
$N_n=\bigoplus_{F_n} H$, so that $N$ is the increasing union of the $N_n$'s. Since
$N_n$ has property (T), we have $\overline{H^1}(N_n,\pi\vert_{N_n})=0$, so by
lemma \ref{inductivelim} we get:
$\overline{H^1}(N,\pi\vert_N)=0$.\\ It remains to show that $V^N=0$. Assume by
contradiction that it is not the case. By irreducibility of $\pi$, this gives $V^N=V$ i.e.
the representation $\pi$ factors through $G/N=\Gamma$. Take $b\in Z^1(G,\pi)$: the restriction
$b\vert_N$ is a homomorphism from $N$ to the additive group of $V$; since $H$ has
property (T), the abelianization group of $H$ is finite, so the
abelianization of $N$ is a torsion group, and this implies $b\vert_N=0$. This shows that $b$ factors through $G/N=\Gamma$, i.e that not only the linear action of $G$ factors through $\Gamma$, but also the affine action associated to $b$. As a consequence $\overline{H^1}(G,\pi)=\overline{H^1}(\Gamma,\pi)$.  Since $\Gamma$
has property (T), we
have $\overline{H^1}(\Gamma,\pi)=0$, and this contradicts our choice of $\pi$.
\hfill$\square$

\begin{Ex} In the above Theorem, $\Gamma$ is assumed to be infinite,
but $H$ is allowed to be finite. So the simplest example of a wreath
product satisfying the assumptions of Theorem \ref{wreath} is
$\mathbb{Z}/2\mathbb{Z}\wr SL_{3}(\mathbb{Z})$.
\end{Ex}


\bibliography{FVbib}
\bibliographystyle{alpha}


\vspace{2cm}
\noindent Talia Fern\'os\\
Einstein Institute of Mathematics, Hebrew University\\
Jerusalem, 91904, Israel\\
E-mail: \url{fernos@math.huji.ac.il}
\vspace{1cm}

\noindent Alain Valette\\
Institut de Mathématiques - Université de Neuchâtel\\
Rue Emile Argand 11, CH-2007 Neuchâtel - Switzerland\\
E-mail: \url{alain.valette@unine.ch}
\vspace{1cm}

\noindent Florian Martin\\
Philip Morris International - R\&D Department\\
CH-2000 Neuch\^atel - Switzerland\\
E-mail: \url{florian.martin@pmintl.com}

\end{document}